\documentclass[letterpaper, 10 pt, conference]{ieeeconf}  % Comment this line out if you need a4paper

\IEEEoverridecommandlockouts                              % This command is only needed if 
\usepackage{graphics} % for pdf, bitmapped graphics files
\usepackage{epsfig} % for postscript graphics files
\usepackage{times} % assumes new font selection scheme installed
\usepackage{amsmath} % assumes amsmath package installed
\usepackage{amssymb}  % assumes amsmath package installed
\usepackage{amsthm}
\usepackage{mathtools}
\usepackage{graphicx}
\usepackage{amsfonts}
\usepackage{xcolor}
\usepackage{algpseudocode,algorithm,algorithmicx}
\usepackage[noadjust]{cite}

\newtheorem{as}{Assumption}
\newtheorem{lm}{Lemma}
\newtheorem{df}{Definition}
\newtheorem{rmk}{Remark}
\newtheorem{thm}{Theorem}

\newcommand{\norm}[1]{\left\lVert #1 \right\rVert}

\title{\LARGE \bf
Asynchronous Zeroth-Order Distributed Optimization \\ with Residual Feedback
}

\author{Yi Shen, Yan Zhang, Scott Nivison, Zachary I. Bell and Michael M. Zavlanos
\thanks{*This work is supported in part by AFOSR under award \#FA9550-19-1-0169
and by NSF under award CNS-1932011.}% <-this % stops a space
\thanks{Yi Shen, Yan Zhang and Michael M. Zavlanos are with the Department of Mechanical Engineering and Materials Science, Duke University, Durham, NC, USA. Email: \{yi.shen478, yan.zhang2, michael.zavlanos\}@duke.edu}%
\thanks{Scott Nivison and  Zachary I. Bell are with the Air Force Research Laboratory, Eglin AFB, FL, USA. Email: \{scott.nivison, zachary.bell.10\}@us.af.mil.}%
}

\begin{document}

\maketitle
\thispagestyle{empty}
\pagestyle{empty}

%%%%%%%%%%%%%%%%%%%%%%%%%%%%%%%%%%%%%%%%%%%%%%%%%%%%%%%%%%%%%%%%%%%%%%%%%%%%%%%%

\begin{abstract}
    We consider a zeroth-order distributed optimization problem, where the global objective function is a black-box function and, as such, its gradient information is inaccessible to the local agents. Instead, the local agents can only use the values of the objective function to estimate the gradient and update their local decision variables. In this paper, we also assume that these updates are done asynchronously. To solve this problem, we propose an asynchronous zeroth-order distributed optimization method that relies on a one-point residual feedback to estimate the unknown gradient. We show that this estimator is unbiased under asynchronous updating, and theoretically analyze the convergence of the proposed method. We also present numerical experiments that demonstrate that our method outperforms two-point methods under asynchronous updating. To the best of our knowledge, this is the first asynchronous zeroth-order distributed optimization method that is also supported by theoretical guarantees.
\end{abstract}

%%%%%%%%%%%%%%%%%%%%%%%%%%%%%%%%%%%%%%%%%%%%%%%%%%%%%%%%%%%%%%%%%%%%%%%%%%%%%%%%

\section{INTRODUCTION}
Distributed optimization algorithms have been used to solve decision making problems in a wide range of application domains, including distributed machine learning~\cite{konevcny2016federated,zhang2019distributed}, resource allocation~\cite{schmidt2009distributed} and robotics~\cite{raffard2004distributed}, to name a few. In these problems, agents aim to find their local optimal decisions so that a global cost function that depends on the joint decisions is minimized. 
Existing algorithms, e.g.,~\cite{boyd2011distributed,nedic2009distributed,chatzipanagiotis2015augmented,zhang2018consensus} often assume that the gradients of the objective function is known and available to the agents. However, this is not always the case in practice. For example, complex systems are often difficult to model explicitly~\cite{bergstra2012random,tu2019autozoom}. Similarly, in applications such as online marketing~\cite{bertsimas2007learning} and multi-agent games~\cite{mertikopoulos2019learning}, the decisions of other agents cannot be observed and, therefore, the gradient of the objective function cannot be locally computed. Finally, many distributed optimization algorithms assume that all agents update their local decisions at the
same time, which requires synchronization over the whole network and can be expensive to implement.

In this paper, we consider distributed optimization problems where a group of agents collaboratively minimize a common cost function that depends on their joint decisions. Moreover, we assume that the agents can only observe and update their local decision variables, and that the gradient of the common objective function with respect to each agent's local decision is not accessible. To solve this problem, zeroth-order optimization methods \cite{flaxman2004online,nesterov2017random,duchi2015optimal,zhang2020improving} have been proposed that estimate the gradient using the values of the objective function. Existing zeroth-order gradient estimators can be classified into two categories, namely, one-point feedback~\cite{flaxman2004online,zhang2020improving} and two-point feedback~\cite{nesterov2017random,duchi2015optimal} estimators, depending on the number of decision points they query at each iteration. The first one-point gradient estimator is analyzed in~\cite{flaxman2004online} and has the form
\begin{equation}\label{one_point}
    G_{{\mu}}(x_k)=\frac{f(x_k+\mu u_k)}{\mu}u_k,
\end{equation}
where  $f$ is a cost function, $x_k$ is a decision variable at time $k$,
$\mu$ is a smoothing parameter, $u$ is sampled from a normal Gaussian distribution $\mathcal{N}(0,I_n)$ and $I_n$ is an $n$ dimensional identity matrix. The estimator~\eqref{one_point} requires to evaluate the objective function at a single point $x_k+\mu u_k$ at iteration $k$ but usually suffers large variance which slows down the optimization process. To reduce the variance of the one-point gradient estimator~\eqref{one_point}, the works in \cite{nesterov2017random,duchi2015optimal} study the two-point gradient estimators
\begin{equation}\label{two_point_unbiased}
\begin{split}
     G_{\mu}(x_k)=\frac{f(x_k+\mu u_k)-f(x_k)}{\mu}u_k  \\
\end{split}
\end{equation}
\begin{equation}\label{two_point_biased}
\begin{split}
   \text{and } \; G_{\mu}(x_k)=\frac{f(x_k+\mu u_k)-f(x_k-\mu u_k)}{2\mu}u_k,
\end{split}
\end{equation}
that evaluate the objective function at two distinct decision points at iteration $k$. Recently, a new one-point gradient estimator has been proposed in \cite{zhang2020improving}, called the residual-feedback gradient estimator,
\begin{equation}\label{residual-feedback}
    \begin{split}
      G_{\mu}(x_k) = \frac{f(x_k+\mu u_k)-f(x_{k-1}+\mu u_{k-1})}{\mu}u_k
    \end{split}
\end{equation}
that enjoys the same variance reduction effect of the two-point gradient estimators~\eqref{two_point_unbiased} and \eqref{two_point_biased} but only requires to evaluate the objective function at a single decision point at each iteration. 
We note that verbatim application of the above centralized gradient estimators to the distributed problem considered in this paper requires synchronization across the agents. This is because the perturbation is according to the full decision vector and required to be implemented simultaneously. To the best of our knowledge, asynchronous zeroth-order distributed optimization methods have not been studied in the literature. If the objective function gradient is known, asynchronous distributed optimization methods with a common objective function have been studied in \cite{tsitsiklis1986distributed,agarwal2012distributed,zhong2008asynchronous}. However, these works can not be directly extended to solve the black-box optimization problems considered here.

In this paper, we consider a model for asynchrony where a single agent is randomly activated at each time step to query the objective function value at one decision point or update its local decision variable. Then, we propose an asynchronous zeroth-order distributed optimization algorithm that relies on extending the centralized residual-feedback gradient estimator~\eqref{residual-feedback} so that it can handle asynchronous queries and updates. Specifically, we show that the proposed zeroth-order gradient estimator provides an unbiased estimate of the gradient with respect to each agent's local decision. Also, we provide bounds on the second moment of this estimator, the first of their kind for any asynchronous zeroth-order gradient of this type, which we then use to show convergence of the proposed method. As expected, the variance of this estimator is greater than he variance of  the estimator in~\eqref{residual-feedback}, since other agents in the network can update their decision variables between an agent's two consecutive queries of the objective function.

We note that, almost concurrently with this work~\cite{cai2021zeroth} proposed an asynchronous zeroth-order optimization algorithm that relies on the two-point gradient estimator~\eqref{two_point_biased}. Unlike the method proposed here, \cite{cai2021zeroth} assumes that while a single agent queries the values of the objective function at decision points $x_k + \mu u_k$ and $x_k$ (or $x_k - \mu u_k$), the other agents do not update their decision variables, even if any of them are activated. This assumption limits the number of updates the agents can make during a fixed period of time and affects the performance of the asynchronous system. Related is also work on distributed zeroth-order methods for the optimization of functions that are the sum of local objective functions, see~\cite{hajinezhad2017zeroth,hajinezhad2018gradient,zhang2020cooperative,tang2020zeroth}. 
However, these methods assume synchronous updates.

The rest of this paper is organized as follows. In Section~\ref{formulation}, we formulate the problem under consideration and present preliminary results on zeroth-order gradient estimators. In Section~\ref{algorith_proof}, we present the proposed asynchronous zeroth-order distributed optimization algorithm with residual feedback gradient estimation, and analyze its convergence. In Section~\ref{simulation}, we numerically validate the proposed algorithm and in Section~\ref{conclusion} we conclude the paper.

\section{PROBLEM FORMULATION AND PRELIMINARIES}\label{formulation}

% example.
%
Consider a multi-agent system consisting of $N$ agents that collaboratively solve the unconstrained optimization problem
\begin{align}
    \min_{x} \; f(x),
\end{align}
where the cost function $f$ is non-convex and smooth, $x:=(x^1,\ldots,x^N)\in \mathbb{R}^{n}$ is the joint decision vector, and $x^i \in  \mathbb{R}^{n_i}$ is the local decision vector of agent $i\in\{1,\dots,N\}$. We first make the following assumptions on the cost function $f$.
\begin{as}\label{lipschitz}
The cost function $f(x):\mathbb{R}^n\to \mathbb{R}$ is bounded below by $f^*$. It is $L_0$-Lipschitz and $L_1$-smooth, i.e.,
\begin{equation*}
    \begin{split}
        |f(x)-f(y)|&\leq L_0\norm{x-y}, \\
        \norm{\nabla f(x)-\nabla f(y)} & \leq L_1\norm{x-y},
    \end{split}
\end{equation*}
for all $x,y\in\mathbb{R}^n.$
\end{as}
As shown in~\cite{nesterov2017random}, $L_1$-smoothness is equivalent to the condition;
\begin{equation}\label{equivalent_smooth}
    |f(y)-f(x)-\langle \nabla f(x),y-x \rangle| \leq \frac{1}{2}L_1\norm{x-y}^2,
\end{equation}
for all $x,y\in\mathbb{R}^n$. For each agent $i$, define the local smoothing function:
\begin{equation}\label{local_smooth}
    f_{\mu_i}(x)=\frac{1}{\kappa_i}\int f(x+\mu_i u_i )e^{-\frac{1}{2}\norm{u_i}^2} du^i_i,
\end{equation}
where $\kappa_i=\int e^{-\frac{1}{2}\norm{u_i}^2} du^i_i$. The random sampling vector $u_{i}=\{u^1_i,\dots,u^N_i\}\in \mathbb{R}^n$ is a vector of all zeros except for the entry $u^i_i$ that is sampled from $\mathcal{N}(0,I_{n_i}).$\footnote{In the following analysis, we drop the agent index $i$ in $u_{i}$ for simplicity.}
Note that $f_{\mu_i}$ preserves all the Lipschitz conditions of $f$ as proved in \cite{nesterov2017random}. Specifically, we have the following lemma.
\begin{lm}
Under Assumption~$\ref{lipschitz}$, we have that , for all agents $i$, $f_{\mu_i}(x):\mathbb{R}^n\to \mathbb{R}$ is $L_0$-Lipschitz and $L_1$-smooth.
\end{lm}
As a result, (\ref{equivalent_smooth}) also holds for $f_{\mu_i}$, which allows us to bound the approximation errors of $f_{\mu_i}$ and $\nabla_i f_{\mu_i}$ with respect to (w.r.t.) to $f$ and $\nabla_i f,$ as shown in Lemma 2 below that is adopted from~\cite{nesterov2017random}. In Lemma 2 and the following analysis, we denote by $\nabla f(x)\in\mathbb{R}^n$ the gradient of $f(x)$. Moreover, we define by $\nabla_i f(x)\in\mathbb{R}^n$ the projection of $\nabla f(x)$ onto the index $i$ by setting the entries of $\nabla f(x)$ not equal to $i$ to be 0. 
\begin{lm}\label{approximation_errors}
Under Assumption~\ref{lipschitz}, the cost function $f$ and its corresponding smoothed function $f_{\mu_i}$ satisfy  $\forall i\in\{1,\dots,N\}$  
\begin{equation}\label{function_error}
 |f_{\mu_i}(x)-f(x)|\leq \frac{\mu_i^2}{2}L_1n_i,
\end{equation}
\begin{equation}\label{gradient_error}
    \norm{\nabla_i f_{\mu_i}(x)-\nabla_i f(x)}\leq \frac{\mu_i}{2}L_1 (n_i+3)^{3/2}.
\end{equation}
\end{lm}
% \begin{proof}
% We first observe that $f_{\mu_i}(x)$ and $f(x)$ only differ at block $i$. By the properties of Gaussian smoothing, we have that 
% %
% \begin{equation*}
%     \begin{split}
% & |f_{\mu_i}(x)-f(x)| \\
% & =\frac{1}{\kappa_i}\int \left[f(x+\mu_iu)-f(x)-\mu_i\langle \nabla f(x), u\rangle\right]e^{-\frac{1}{2}\norm{u}^2} du^i \\
% & \leq \frac{\mu_i^2L_1}{2\kappa_i}\int \norm{u}^2 e^{-\frac{1}{2}\norm{u}^2} du^i \leq \frac{\mu_i^2}{2}L_1n_i,\\
%     \end{split} 
% \end{equation*}
% where the first inequality holds due to~(\ref{equivalent_smooth}) and the second inequality is proved in~\cite{nesterov2017random}. Equation~(\ref{gradient_error}) can be proved similarly as above, 
Next, we define the model that the agents use to asynchronously update their decision variables. 
\begin{df}[Asynchrony Model]\label{as}
At each time step, one agent is independently and randomly selected according to a fixed distribution $P = [p_1, \dots, p_N]$. The selected agent $i$ can query the value of the cost function\footnote{Here, we assume that each agent receives noiseless feedback $f(x)$. The proposed method can be extended to noisy feedback with bounded variance.} once and update its local decision variable, while the decisions of the other agents $\{x^j\}_{j\neq i}$ are fixed. The agents only communicate with a central entity that has access to the global cost function but not with each other.
% {\color{purple}While being selected, the agent can query the value of the cost function $f$\footnote{Here, we assume that each agent receives noiseless feedback $f(x)$. The proposed method can be extended to noisy feedbacks with bounded variance.} once at any point $x^i\in\mathbb{R}^{n_i}$ but not others, i.e., $\{x^j\}_{j\neq i}$ are fixed. After each query, the agent can either wait for the next selection or update its local decisions immediately.} 
\end{df}

\begin{rmk}\label{remark1}
Definition~\ref{as} can be satisfied when all agents make queries and update their decisions according to their local clock without any central coordination. Specifically, let the time interval between each agent's consecutive queries be called a waiting time. Then, if each local waiting time is random and subject to an exponential distribution, according to Chapter 2.1 in~\cite{durrett1999essentials}, Definition~\ref{as} will be satisfied.
\end{rmk}

%
%%%%%%%%%%%%%%%%%%%%%%%%%%%%%%%%%%%%%%%%%%%%%%%%%%%%%%%%%%%%%%%%%%%%%%%%%%%%%%%%
\section{ALGORITHM DESIGN AND ANALYSIS}\label{algorith_proof}

In this section, we present the proposed asynchronous zeroth-order distributed optimization algorithm and analyze its convergence rate.
To do so, we first propose an asynchronous zeroth-order gradient estimator based on the centralized residual feedback estimator~\eqref{residual-feedback}. For every agent $i$, at time step $k$, this estimator takes the form
\begin{equation}\label{residual}
     G_{\mu_i}(x_k) =  \frac{f(x_k+\mu_i u_{k})-f(x_{k-M}+\mu_i u_{k-M})}{\mu_i}u_{k},
\end{equation}
where $k-M$ is a random index denoting the iteration when agent $i$ conducted its most recent update. This index takes values on a global time scale. The random sampling vector $u_{k}$ is as defined in~(\ref{local_smooth}).
Note that $G_{\mu_i}$ is different from the centralized zeroth-order gradient estimator~\eqref{residual-feedback}, where $u_k$ is a perturbation along the full decision vector $x$.
Here, $G_{\mu_i}$ estimates the gradient by perturbing the function $f$ along a random direction restricted to agent $i$'s block of the full decision vector $x$ and uses the previous query to reduce the variance. 
Indeed, $G_{\mu_i}$ provides an unbiased gradient estimate of the corresponding smoothed function $f_{\mu_i}$ restricted to agent $i$'s block, as shown in the following lemma.  
% The local smoothing function $f_{\mu_i}$ is defined as:
% %
% %
% \begin{equation}\label{local_smooth}
%     f_{\mu_i}(x)=\frac{1}{\kappa_i}\int f(x+\mu_i u )e^{-\frac{1}{2}\norm{u}^2} du^i,
% \end{equation}
% %
% where $\kappa_i=\int e^{-\frac{1}{2}\norm{u}^2} du^i$.
% %
%

%

\begin{lm}\label{unbiased_gradient}
For each agent $i$, we have that  
$$\mathbb{E}\left[G_{\mu_i}(x_k)\right]=\nabla_i f_{\mu_i}(x_k).$$
\end{lm}
\begin{proof}
Taking the expectation of both sides of~\eqref{residual}, we obtain that 
\begin{equation*}
\begin{split}
     \mathbb{E}\left[G_{\mu_i}(x_k)\right] & =  \mathbb{E}\left[\frac{f(x_k+\mu_i u_k)-f(x_{k-M}+\mu_i u_{k-M})}{\mu_i}u_k\right] \\
    & = \mathbb{E}\left[\frac{f(x+\mu_i u_k)}{\mu_i}u_k\right]= \nabla_i f_{\mu_i}(x_k),  
\end{split}
\end{equation*}
where the second equality follows from the fact that $x_{k-M}$ and $u_{k-M}$ are independent from $u_k$ and the expectation of $u_k$ is 0 . The last equality follows from the definitions of $f_{\mu_i}$ and $\nabla_i f_{\mu_i}$. 
\end{proof}
\begin{rmk}
Note that both $G_{\mu_i}(x_k)$ and $\nabla_i f_{\mu_i}$ are vectors in $\mathbb{R}^n$ with entries equal to zero at blocks other than $i$.
\end{rmk}
Using the local gradient estimate $G_{\mu_i}(x_k)$, we can define the update rule for every agent i as 
\begin{equation}\label{updating_rule}
    x_{k+1}=x_k-\alpha_i G_{\mu_i}(x_k),
\end{equation} where $\alpha_i$ is the step size.
The proposed asynchronous zeroth-order distributed optimization algorithm with residual feedback is described in Algorithm~\ref{alg:B}\footnote{Note that, we can also extend~(\ref{two_point_unbiased}) for asynchronous problems and design the algorithm thereof. The lemmas and theorems proved in this paper can be easily adapted to this case as well. In Section~\ref{simulation}, we will compare these two gradient estimators empirically. However, the extension of~(\ref{two_point_biased}) is non-trivial. It can be verified that Lemma~\ref{unbiased_gradient} does not hold for the extension of~(\ref{two_point_biased}). We leave it as future work.}.

\begin{algorithm}[t]
	\caption{Asynchronous Zeroth-Order Residual Feedback}
	\label{alg:B}
	\begin{algorithmic}[1]
		\Require{sampling rate $p_i$ with $\sum_{i=1}^N p_i =1$,  decision variable $x^i_0$, smoothing parameter $\mu_i$ and step size $\alpha_i$ for all agents $i$. Set the iteration counter $t=0$ and let $T$ be the maximum number of iterations.}
		\For{$t\leq T$}
		\State{sample an index $i_t$ according to $\mathbb{P}(i_t=i)=p_i$}
		\State{sample $u^i\sim \mathcal{N}(0,I_{n_i})$}
		\State{query the function value $f(x+\mu_iu)$}
		\State{compute $G_{\mu_i}$ according to~(\ref{residual})} 
		\State{update local decision $x^i\leftarrow x^i-\alpha_iG_{\mu_i}(x^i)$}
		\State{update the time step counter $t\leftarrow t+1$}
		\EndFor
	\end{algorithmic}
\end{algorithm}
%

%

%
%%%%%%%%%%%%%%%%%

Without loss of generality, we assume that decision variables $x^i\in \mathbb{R}^{\Bar{n}}$ of all agents $i$ have the same dimensions, and that the step sizes and smoothing parameters of all agents are also the same, i.e.,  $\alpha_i=\alpha$ and $\mu_i=\mu$.
To analyze the convergence of Algorithm~\ref{alg:B}, we need to bound the second moment of the proposed gradient estimator $G_{\mu_i}(x_k)$.
However, under the asynchronous framework considered in this paper, there can be a random number of agents updating their local decision variables between the two queries made by agent $i$ at time steps $k$ and $k-M$ in \eqref{residual}. These updates by other agents introduce additional variance into the estimator~\eqref{residual} compared to the variance of the centralized estimator analyzed in \cite{zhang2020improving}. Next, we analyze the effect of the asynchronous updates on the second moment of the zeroth-order gradient estimator. To the best of our knowledge, this is the first time that a bound on the second moment of a zeroth-order gradient estimator is provided for asynchronous problems. An additional contribution of this work, is that the proof technique presented below can be extended to obtain similar results for the two-point gradient estimator~\eqref{two_point_unbiased}.

% Previous analyses in~\cite{zhang2020improving} cannot be directly applied to the asynchronous case and the extension is not trivial since in~(\ref{residual}) $k-M$ is random. The following lemma shows that the second moment of $G_{\mu_i}(x_k)$ can be bounded by a linear combination of previous second moments $\{G_{\mu_i}(x_m)\}_{m=0}^{k-1}$.  
%%%%%%%%%%%%%%%%%
\begin{lm}\label{lemma_second_moments}
Let Assumptions~\ref{lipschitz} hold under the Asynchrony Model, and define by $\mathbb{E}[\norm{G_{\tilde{\mu}}(x_k)}^2]:=\mathbb{E}_{i_k} [\mathbb{E}_{u_{[k]},i_{[k-1]}}$ $[\norm{G_{\mu_i}(x_k)}^2|i_k=i]]$, where $u_{[k]}=(u_1,\dots,u_k)$ and $i_{[k]}=(i_1,\dots,i_k),$ Then, running the asynchronous Algorithm~\ref{alg:B}, we have that $\mathbb{E}[\norm{G_{\tilde{\mu}}(x_k)}^2]$ satisfies
\begin{equation}\label{equation_lemma_4_recursive}
\begin{split}
 \mathbb{E} & \left[\norm{G_{\tilde{\mu}}(x_k)}^2\right] \\
 \leq &  \frac{2\Bar{n}L_0^2\alpha^2 k}{\mu^2} \sum_{m=0}^{k-1} (1-p_{\min})^{m} \mathbb{E}\left[\norm{G_{\tilde{\mu}}(x_{k-m-1})}^2\right] \\
 &+ 4L_0^2\left((4+\Bar{n})^2+\Bar{n}^2\right), \\
\end{split}
\end{equation}
where $p_{\min}=\min_{i}p_i$, and the expectations are taken w.r.t. the sequence of random exploration directions $\{u_k\}$ and the sequence of random indices of activated agents $\{i_k\}$. 
\end{lm}
% \begin{proof}
% \end{proof}
%%%%%%%%%%%%%%%%%%%%%%%%%%%%%%%%%%%%%%%%%%%%%%%%%%%%%%%%%
\begin{proof}
%Denote $u_{[k]}=(u_1,\dots,u_k),i_{[k]}=(i_1,\dots,i_k).$
%
Suppose that at time step $k$, agent $i$ is selected. Taking the second moment of $G_{\mu_i}$ and using equation~\eqref{residual}, we we have\footnote{To simplify the notation, when it is clear from the context, we drop the subscript of the expectation and the conditional event, e.g., $\mathbb{E}[\norm{G_{\mu_i}(x_k)}^2]:=\mathbb{E}_{u_{[k]},i_{[k-1]}}[\norm{G_{\mu_i}(x_k)|i_k=i}^2]$.}
%{\color{red}Give example of notations which you actually ignore the condition signs, e.g., $\|x_k - x_{k-M}\|$.}} 
\begin{equation}\label{second_moment_ineqaulity}
\begin{split}
 & \mathbb{E}_{u_{[k]},i_{[k-1]}}\left[\norm{G_{\mu_i}(x_k)}^2|i_k=i\right] \leq\\
 &  \mathbb{E}\left[\frac{\left(f(x_k+\mu_iu_{k})-f(x_{k-M}+\mu_iu_{k-M})\right)^2\norm{u_{k}}^2}{\mu_i^2}\right]  
\end{split}
\end{equation}
Notice that 
\begin{equation}\label{ab}
    \begin{split}
      & \left(f(x_k+\mu_iu_{k})-f(x_{k-M}+\mu_iu_{k-M})\right)^2 \\
      &\leq 2 \underbrace{ \left(f(x_k+\mu_iu_{k})-f(x_{k-M}+\mu_iu_{k})\right)^2}_\text{a} \\
      &+ 2  \underbrace{\left(f(x_{k-M}+\mu_iu_{k})-f(x_{k-M}+\mu_iu_{k-M})\right)^2}_\text{b}.
    \end{split}
\end{equation}
Substituting~(\ref{ab}) into~({\ref{second_moment_ineqaulity}}), we obtain that  
\begin{equation}\label{a+b_inequality}
\begin{split}
 & \mathbb{E}_{u_{[k]},i_{[k-1]}}\left[\norm{G_{\mu_i}(x_k)}^2|i_k=i\right]\\
 & \leq \mathbb{E}\left[\frac{2a+2b}{\mu_i^2}\norm{u_{k}}^2\right]\\
 & \leq \mathbb{E}\left[\frac{2L_0^2\norm{x_k-x_{k-M}}^2}{\mu_i^2}\norm{u_{k}}^2\right]+\mathbb{E}\left[\frac{2b}{\mu_i^2}\norm{u_{k}}^2\right]\\
 & \leq \frac{2L_0^2\Bar{n}}{\mu_i^2}\mathbb{E}\left[\norm{x_k-x_{k-M}}^2\right]+\mathbb{E}\left[\frac{2b}{\mu_i^2}\norm{u_{k}}^2\right],\\
\end{split}
\end{equation}
where the second inequality holds due to Lipschitzness of $f$ and the last inequality holds since $x_k-x_{k-M}$ is independent from $u_k$ and $\mathbb{E}[\norm{u_k}^2]=\Bar{n}.$ 
We first bound the second term in the right-hand-side of~\eqref{a+b_inequality}. Specifically, we have that 
\begin{equation}\label{b_inequality}
    \begin{split}
        &\mathbb{E}\left[\frac{2b}{\mu_i^2}\norm{u_{k}}^2\right] \leq  \mathbb{E}\left[2L_0^2\norm{u_{k}- u_{k-M}}^2\norm{u_{k}}^2\right]\\
        & \leq \mathbb{E}\left[4L_0^2\left(\norm{u_{k}}^2+\norm{u_{k-M}}^2\right)\norm{u_{k}}^2\right]\\
      & \leq \mathbb{E}\left[4L_0^2\norm{u_{k}}^4\right]+\mathbb{E}\left[4L_0^2\norm{u_{k-M}}^2\right] \mathbb{E}\left[\norm{u_{k}}^2\right]\\
      & \leq 4L_0^2\left((4+\Bar{n})^2+\Bar{n}^2\right),
    \end{split}
\end{equation}
where the first inequality holds due to Lipschitzness of $f$, the third inequality holds since $u_{k-M}$ is independent from $u_k$ and the last inequality follows from Lemma 1 in~\cite{nesterov2017random}. 

Next, we bound the first term in the right-hand-side of~\eqref{a+b_inequality} containing the second moment of $x_k-x_{k-M}.$
Given that agent $i$ updates at time step $k$, we can partition the sequence of all past updates into $k$ events $A^i_m=\{M=m\}$, where $A^i_m$ represents all sequences of updates such that the most recent update by agent $i$ is at global time step $k-m$. In particular, $A^i_k$ indicates that agent $i$ has not been updated before and $k$ is the first time that this agent gets updated. It is easy to see that the sets $\{A^i_m\}_{m=1}^k$ are disjoint and contain all possible 
sequences of updates by the team of agents.
Using the definition of these events, we can rewrite the conditional expectation of $\norm{x_k-x_{k-M}}^2$ as
\begin{equation}\label{k_k-m}
    \begin{split}
        & Y_i := \mathbb{E}_{u_{[k]},i_{[k-1]}}\left[\norm{x_k-x_{k-M}}^2|i_k=i\right]\\
        & = \sum_{m=1}^k \mathbb{E}_{u_{[k]}}\left[\norm{x_k-x_{k-m}}^2|A^i_m,i_k=i\right]\mathbb{P}(A^i_m).
    \end{split}
\end{equation}
Equation~(\ref{k_k-m}) can be rewritten as 
\begin{align*}
      Y_i  = & \sum_{m=1}^k\mathbb{E}\left[\norm{x_k-x_{k-m}}^2|A^i_m\right]\mathbb{P}(A^i_m) \\
      = &\sum_{m=1}^k\mathbb{E}\left[\norm{\sum_{l=0}^{m-1}(x_{k-l}-x_{k-l-1})}^2|A^i_m\right]\mathbb{P}(A^i_m)  \\
    \leq & \sum_{m=1}^k\mathbb{E}\left[m\sum_{l=0}^{m-1}\norm{(x_{k-l}-x_{k-l-1})}^2|A^i_m\right]\mathbb{P}(A^i_m), \\
\end{align*}
where the last inequality holds due to the fact that $\left(\sum_{m=1}^k a_m\right)^2\leq k\sum_{m=1}^k a_m^2.$
We first collect all the terms containing $\norm{x_k-x_{k-1}}^2$, which we denote by $\{Y_i:\norm{x_k-x_{k-1}}^2\}$. Then, we have that
\begin{align}
        &\{Y_i:\norm{x_k-x_{k-1}}^2\}
        = \sum_{m=1}^k m\mathbb{E}\left[\norm{x_k-x_{k-1}}^2|A^i_m\right]\mathbb{P}(A^i_m) \nonumber\\
        &\leq  k\sum_{m=1}^k \mathbb{E}\left[\norm{x_k-x_{k-1}}^2|A^i_m\right]\mathbb{P}(A^i_m) \nonumber \\
        &= k\mathbb{E} \left[\norm{x_k-x_{k-1}}^2\right]\label{yk},
\end{align}
where the last equation holds by the definition of conditional expectation.
Next we collect all the terms containing $\norm{x_{k-s}-x_{k-s-1}}^2$ for all $s\in\{1,\dots,k-1\}$. Specifically, we have that 
\begin{equation}\label{ys}
    \begin{split}
        & \{Y_i:\norm{x_{k-{s}}-x_{k-s-1}}^2\} \\
        &=  \sum_{m=s+1}^k m\mathbb{E}\left[\norm{x_{k-{s}}-x_{k-s-1}}^2|A^i_m\right]\mathbb{P}(A^i_m) \\
        & \leq k\sum_{m=s+1}^k  \mathbb{E}\left[\norm{x_{k-{s}}-x_{k-s-1}}^2|A^i_m\right]\mathbb{P}(A^i_m). \\
        % & = k \mathbb{P}(A^{ic}_{1,\dots,m})\mathbb{E}\left[\norm{x_{k-{s}}-x_{k-s-1}}^2\right] ,
    \end{split}
\end{equation}
% where $A^{ic}_{1,\dots,m}=\cup_{m=s+1}^k A^i_m$.
%
We claim that the right hand side of~\eqref{ys} satisfies the following equation
\begin{equation}\label{two_identity}
    \begin{split}
      &k\sum_{m=s+1}^k  \mathbb{E}\left[\norm{x_{k-{s}}-x_{k-s-1}}^2|A^i_m\right]\mathbb{P}(A^i_m)\\
      & =k\mathbb{P}(A^{ic}_{1:s})\mathbb{E}\left[\norm{x_{k-{s}}-x_{k-s-1}}^2\right],
    \end{split}
\end{equation}
where $A^{ic}_{1:s}=\cup_{m=s+1}^k A^i_m$. To see this, we first observe that
\begin{equation}\label{two_quantity}
    \begin{split}
  \mathbb{E}&\left[\norm{x_{k-{s}}-x_{k-s-1}}^2\right] \\
  = & \sum_{m=1}^k \mathbb{E}\left[\norm{x_{k-{s}}-x_{k-s-1}}^2|A^i_m\right]\mathbb{P}(A^i_m)\\
    =& \mathbb{E}\left[\norm{x_{k-{s}}-x_{k-s-1}}^2|A^{i}_{1:s}\right]\mathbb{P}(A^{i}_{1:s})\\
    &+\mathbb{E}\left[\norm{x_{k-{s}}-x_{k-s-1}}^2|A^{ic}_{1:s}\right]\mathbb{P}(A^{ic}_{1:s}),\\
    \end{split}
\end{equation}
% \begin{equation*}
%     \begin{split}
%       &\mathbb{E}\left[\norm{x_{k-s}-x_{k-s-1}}^2|A^{i}_{1:s}\right],\\ &\mathbb{E}\left[\norm{x_{k-s}-x_{k-s-1}}^2|A^{ic}_{1:s}\right],
%     \end{split}
% \end{equation*}
%
where $A^{i}_{1:s}=\cup_{m=1}^s A^i_m$ and $A^{ic}_{1:s}$ is the complement of $A^{i}_{1:s}$. The second equality follows from the property of conditional expectation of disjoints events.
Note that $\mathbb{E}[\norm{x_{k-{s}}-x_{k-s-1}}^2|A^{i}_{1:s}]$ and $\mathbb{E}[\norm{x_{k-{s}}-x_{k-s-1}}^2|A^{ic}_{1:s}]$ are equal. Specifically, event $A^{i}_{1:s}$ and event $A^{ic}_{1:s}$ only differ after time step $k-s$, where $A^{i}_{1:s}$ contains all sequences of updates where $i$ update after $k-s$ and $A^{ic}_{1:s}$ contains all sequences of updates where $i$ does not updates after $k-s$. Since both events $A^{i}_{1:s}$ and $A^{ic}_{1:s}$ do not affect the agents' updates before time step $k-s$, we have that  
\begin{equation*}
    \begin{split}
  \mathbb{E}&\left[\norm{x_{k-{s}}-x_{k-s-1}}^2\right] 
%   =&  \sum_{m=1}^k \mathbb{E}\left[\norm{x_{k-{s}}-x_{k-s-1}}^2|A^i_m\right]\mathbb{P}(A^i_m)\\
%   %
%     =& \mathbb{E}\left[\norm{x_{k-{s}}-x_{k-s-1}}^2|A^{i}_{1:s}\right]\mathbb{P}(A^{i}_{1:s})\\
%     &+\mathbb{E}\left[\norm{x_{k-{s}}-x_{k-s-1}}^2|A^{ic}_{1:s}\right]\mathbb{P}(A^{ic}_{1:s})\\
%     %
    = \mathbb{E}\left[\norm{x_{k-{s}}-x_{k-s-1}}^2|A^{ic}_{1:s}\right].\\
    \end{split}
\end{equation*}
Combining the above equality with~(\ref{ys}), we have that
\begin{equation*}
    \begin{split}
      &k\sum_{m=s+1}^k  \mathbb{E}\left[\norm{x_{k-{s}}-x_{k-s-1}}^2|A^i_m\right]\mathbb{P}(A^i_m)\\
      & =k\mathbb{P}(A^{ic}_{1:s})\mathbb{E}\left[\norm{x_{k-{s}}-x_{k-s-1}}^2|A^{ic}_{1:s}\right]\\
      & =k\mathbb{P}(A^{ic}_{1:s})\mathbb{E}\left[\norm{x_{k-{s}}-x_{k-s-1}}^2\right],
    \end{split}
\end{equation*}
which completes the proof of~(\ref{two_identity}).
% %
% \begin{equation*}
%     \begin{split}
%      &\mathbb{E}\left[\norm{x_{k-{s}}-x_{k-s-1}}^2|A^{ic}_{1:s}\right]\\
%      &=\frac{\sum_{m=s+1}^k\mathbb{P}(A^i_m)\mathbb{E}\left[\norm{x_{k-{s}}-x_{k-s-1}}^2|A^i_m\right]}{\mathbb{P}(A^{ic}_{1:s})}.   
%     \end{split}
% \end{equation*}
% %
% %
% By the calculation of $Y_i[\norm{x_{k-1}-x_{k-2}}^2]$, we obtain
% %
% %
% \begin{equation*}
%     \begin{split}
%         Y_i&\left[\norm{x_{k-1}-x_{k-2}}^2\right] \\
%         = &\sum_{m=2}^k m\mathbb{E}\left[\norm{x_{k-1}-x_{k-2}}^2|A^i_m\right]\mathbb{P}(A^i_m). \\
%          \leq &k\sum_{m=2}^k \mathbb{E}\left[\norm{x_{k-1}-x_{k-2}}^2|A^i_m\right]\mathbb{P}(A^i_m)\\
%          = & k \mathbb{P}(A^{ic}_1)\mathbb{E}\left[\norm{x_{k-1}-x_{k-2}}^2\right],
%     \end{split}
% \end{equation*}
% where the last equality follows from~(\ref{two_identity}).
% %
% Repeat the above process, we can collect all terms containing $\norm{x_{k-s}-x_{k-s-1}}^2$ and bound it from the right by $\mathbb{E}[\norm{x_{k-s}-x_{k-s-1}}^2] $ for all $s\in\{0,\dots,k-1\}.$ 
% %
% Specifically, we have the following inequalities:
% \begin{equation*}
%     \begin{split}
%         & Y_i\left[\norm{x_{k-{s}}-x_{k-s-1}}^2\right] \\
%         & =  \sum_{m=s+1}^k m\mathbb{E}\left[\norm{x_{k-{s}}-x_{k-s-1}}^2|A^i_m\right]\mathbb{P}(A^i_m) \\
%         %
%         & \leq k\sum_{m=s+1}^k  \mathbb{E}\left[\norm{x_{k-{s}}-x_{k-s-1}}^2\right]\mathbb{P}(A^i_m) \\
%         & = k \mathbb{P}(A^{ic}_{1,\dots,m})\mathbb{E}\left[\norm{x_{k-{s}}-x_{k-s-1}}^2\right] ,
%     \end{split}
% \end{equation*}
% where $A^{ic}_{1,\dots,m}=\cup_{m=s+1}^k A^i_m$.
% %
Then, we can bound $Y_i$ in \eqref{k_k-m} by combining~(\ref{yk}) and~(\ref{two_identity}) and have that
    \begin{align}
         \mathbb{E}&\left[\norm{x_k-x_{k-M}}^2\right]\leq k\mathbb{E}\left[\norm{x_k-x_{k-1}}^2\right] \nonumber\\ 
         &+k\sum_{m=1}^{k-1}\mathbb{P}(A^{ic}_{1:m})\mathbb{E}\left[\norm{x_{k-m}-x_{k-m-1}}^2\right].\label{bounding_k-m}
    \end{align}
By definition, we have 
$\mathbb{P}(A^i_m) = p_i(1-p_i)^{m-1}$ for $1\leq m <k,$ and 
$\mathbb{P}(A^i_k) = (1-p_i)^{k}$ for $m=k$, where $p_i$ is the probability of agent $i$ being sampled at each time step. As a result, $\mathbb{P}(A^{ic}_{1:m})=(1-p_i)^{m}.$
Substituting these probabilities into~(\ref{bounding_k-m}), we have that
\begin{equation}\label{compact_relation}
    \begin{split}
        &\mathbb{E}\left[\norm{x_k-x_{k-M}}^2\right] \\
        % & \leq k \sum_{m=0}^{k-1} (1-p_i)^{m} \mathbb{E}\left[\norm{x_{k-m}-x_{k-m-1}}^2|i_k=i\right]\\
        & \leq k \sum_{m=0}^{k-1} (1-p_{\min})^{m} \mathbb{E}\left[\norm{x_{k-m}-x_{k-m-1}}^2\right],
    \end{split}
\end{equation}where $p_{\min} =\min_i{p_i}$.
Substituting~(\ref{compact_relation}) and~(\ref{b_inequality}) into~(\ref{a+b_inequality}), we get that
\begin{equation}\label{a+b_inequality_final}
\begin{split}
 & \mathbb{E}\left[\norm{G_{\mu_i}(x_k)}^2\right]\\
 & \leq  \frac{2L_0^2\Bar{n}k}{\mu_i^2} \sum_{m=0}^{k-1} (1-p_{\min})^{m} \mathbb{E}\left[\norm{x_{k-m}-x_{k-m-1}}^2\right]\\
 & \quad + 4L_0^2\left((4+\Bar{n})^2+\Bar{n}^2\right).
\end{split}
\end{equation}
Recall that all expectations from the beginning of the proof are taken conditioned on the event $\{i_k=i\}$. Now, taking the expectation w.r.t. $i_k$ on both sides of \eqref{a+b_inequality_final}, and substituting the step size $\alpha$ and smoothing parameter $\mu$ into (\ref{a+b_inequality_final}), we get that
% \begin{equation}\label{compact_relation_1}
%     \begin{split}
%         &\mathbb{E}_{u_{[k]},i_{[k]}}\left[\norm{x_k-x_{k-M}}^2\right] \\
%         & \leq k \sum_{m=0}^{k-1} (1-p_{\min})^{m} \mathbb{E}_{u_{[k]},i_{[k]}}\left[\norm{x_{k-m}-x_{k-m-1}}^2\right].
%     \end{split}
% \end{equation}
% Denote $\mathbb{E}_{u_{[k]},i_{[k]}}[\norm{x_k-x_{k-M}}^2]$ the second moments including the randomness of $i_k$, we have
% %
% \begin{equation*}\label{second_recursive}
%     \begin{split}
%         &\mathbb{E}_{u_{[k]},i_{[k]}}\left[\norm{x_k-x_{k-M}}^2\right]\\ &=\sum_{i=1}^N\mathbb{E}_{u_{[k]},i_{[k-1]}}\left[\norm{x_k-x_{k-M}}^2\right]\mathbb{P}(i_k=i)\\
%         &\leq k \sum_{i=1}^N  \sum_{m=0}^{k-1} (1-p_i)^{m} \mathbb{E}\left[\norm{x_{k-m}-x_{k-m-1}}^2\right]p_i\\
%         &\leq k \sum_{i=1}^N  \sum_{m=0}^{k-1} (1-p_{\min})^{m} \mathbb{E}\left[\norm{x_{k-m}-x_{k-m-1}}^2\right]p_i\\
%         % &= k  \sum_{m=0}^{k-1}  (1-p_{\min})^{m}\sum_{i=1}^N  \mathbb{E}\left[\norm{x_{k-m}-x_{k-m-1}}^2\right]p_i\\
%         &= k  \sum_{m=0}^{k-1}  (1-p_{\min})^{m} \mathbb{E}_{u_{[k]},i_{[k]}}\left[\norm{x_{k-m}-x_{k-m-1}}^2\right],
%     \end{split}
% \end{equation*}
% where the first inequality holds according to~(\ref{compact_relation}). 
% %
% %
% Taking the expectation with respect to (w.r.t.) $i_k$ on both sides of~(\ref{a+b_inequality_final}) and substituting the common step size $\alpha$, smoothing parameter $\mu$ and the above inequality, we obtain the following bound:
% %
\begin{equation*}
\begin{split}
 &\mathbb{E}\left[\norm{G_{\tilde{\mu}}(x_k)}^2\right]\\
%  &=\mathbb{E}_{i_k}\mathbb{E}_{u_{[k]},i_{[k-1]}}\left[\norm{x_k-x_{k-M}}^2|i_k=i\right]\\
%
 &\leq  \frac{2\Bar{n}L_0^2\alpha^2k}{\mu^2}   \sum_{m=0}^{k-1}  (1-p_{\min})^{m} \mathbb{E}\left[\norm{G_{\tilde{\mu}}(x_{k-m-1})}^2\right]\\
&\quad + 4L_0^2\left((4+\Bar{n})^2+\Bar{n}^2\right),
\end{split}
\end{equation*}where $\mathbb{E}[\norm{x_{k-m}-x_{k-m-1}}^2] = \alpha^2 \mathbb{E}[\norm{G_{\tilde{\mu}}(x_{k-m-1})}^2]$ according to the update rule~\eqref{updating_rule} and the definition of $\mathbb{E}[\|G_{\tilde{\mu}}\|^2]$ as in Lemma~\ref{lemma_second_moments}.
% The inequality follows by plugging the updating rule~(\ref{updating_rule}) and the definition of $G_{\tilde{\mu}}$, i.e., for any $m\in\{0,\dots,k-1\}$, we have $\mathbb{E}_{i_{k-m-1}}[\mathbb{E}_{u_{[k-m-1]},i_{[k-m-2]}}\norm{x_{k-m}-x_{k-m-1}}^2|i_{k-m-1}]=\alpha^2\mathbb{E}_{i_{k-m-1}}[\mathbb{E}_{u_{[k-m-1]},i_{[k-m-2]}}\norm{G_{\mu_i}(x_{k-m-1})}^2|i_{k-m-1}]=\alpha^2\mathbb{E}[\norm{G_{\tilde{\mu}}(x_{k-m-1})}^2],$
% %
% which completes the proof. 
The proof is complete. 
\end{proof}
%%%%%%%%%%%%%%%%%%%%%%%%%%%%%%%%%%%%%%%%%%%%%%%%%%%%%%%%%%%%%%%%%%%%

Lemma~\ref{lemma_second_moments} indicates that the second moment of the zeroth-order gradient estimate at time step $k$ is related to the second moments of all previous gradient estimates. Specifically, the effect of the second moment of the past gradient estimates on the current estimate diminishes geometrically over time. 
Next, using Lemma~\ref{lemma_second_moments}, we present a bound on the accumulated second moments of the residual-feedback gradient estimates from $k=0$ to $T-1$, which we will later use to prove our main theorem.
% Indeed, as we will see in the following theorem when summing up all second moments from $k=0$ to $T$, it can be bounded by a multiplication of the second moment of $x_0$. Before presenting the main theorem, we introduce the following lemma that will be used in the analysis.
%%%
\begin{lm}\label{lemma_recursive_G}
Let Assumptions~\ref{lipschitz} hold under the Asynchrony Model. Then, running the asynchronous updates Algorithm~\ref{alg:B}, we have that
%
% \begin{equation*}
%     \begin{split}
%      &\mathbb{E}\left[\norm{G_{\tilde{\mu}}(x_k)}^2\right]\\
%      &\leq \gamma \left(\mathbb{E} \left[\norm{G_{\tilde{\mu}}(x_{k-1})}^2\right]  +\dots + \beta^{k-1} \mathbb{E}\left[\norm{G_{\tilde{\mu}}(x_{0})}^2\right]\right)+M,   
%     \end{split}
% \end{equation*}
% where $M$ is a constant, then we have 
% \begin{equation*}
% \begin{split}
%   \mathbb{E}\left[\norm{G_{\tilde{\mu}}(x_{k})}^2\right] \leq & \gamma(\gamma +\beta)^{k-1}\mathbb{E}\left[\norm{G_{\tilde{\mu}}(x_{0})}^2\right]\\
%  &+\frac{1-\beta -\gamma(\gamma+\beta)^{k-1}}{1-(\gamma+\beta)}M.    
% \end{split}
% \end{equation*}
% In addition,
\begin{equation*}
    \begin{split}
    \sum_{k=0}^{T-1} \mathbb{E}&\left[\norm{G_{\tilde{\mu}}(x_{k})}^2\right] \leq  \frac{1-\beta}{1-(\gamma+\beta)}\mathbb{E}\left[\norm{G_{\tilde{\mu}}(x_{0})}^2\right]\\
    &+(T-1)\frac{1-\beta}{1-(\gamma+\beta)}M - \frac{\gamma}{\left(1-(\gamma+\beta)\right)^2}M,
    \end{split}
\end{equation*}
where $\gamma =\frac{2\Bar{n}L_0^2\alpha^2(T-1)}{\mu^2} $, $\beta=1-p_{\min}$, $M=4L_0^2\left((4+\Bar{n})^2+\Bar{n}^2\right)$ and provided with $0<\gamma+\beta <1.$
\end{lm}
The proof follows from Lemma~\ref{lemma_recursive} in the Appendix.
%

%%%%%%%%%%%%%%%%%%%%%%%%%%%%%%%%%%%%%%%%%%%%%%%%%%%%%%%%%%%
\begin{thm}\label{theorem_1}
Let Assumptions~\ref{lipschitz} hold under the Asynchrony Model. Moreover, run the asynchronous algorithm Algorithm~\ref{alg:B} for $T$ iterations and let $\Tilde{x}$ be uniformly randomly selected from $T$ iterations. Then, selecting the step size $\alpha=\frac{\sqrt{p_{\min}}}{T^{\frac{2}{3}}}$ and the smoothing parameter $\mu=\frac{2L_0\sqrt{\Bar{n}}}{T^{\frac{1}{6}}}$, we have $\mathbb{E}\left[{\norm{\nabla f(\tilde{x})}}^2\right]\leq \mathcal{O}(\Bar{n}^3T^{-\frac{1}{3}}).$
\end{thm}
\begin{proof}
Substituting $x_{k+1}$ and $x_k$ in the version of~\eqref{equivalent_smooth} for the smoothed function $f_{\mu_i}$, we obtain that 
    \begin{align}
    &f_{\mu_i}(x_{k+1}) \nonumber\\
     &\leq  f_{\mu_i}(x_{k}) + \langle\nabla f_{\mu_i}(x_k),x_{k+1}-x_{k}\rangle +\frac{L_1}{2} \norm{x_{k+1}-x_{k}}^2 \nonumber\\
    %  = &f_{\mu_i}(x_{k})-\alpha \langle \nabla_i  f_{\mu_i}(x_k),G_{\mu_i}(x_k) \rangle+\frac{L_1\alpha^2}{2}\norm{G_{\mu_i}(x_k)}^2\\
     &= f_{\mu_i}(x_{k})-\alpha \langle \nabla_i  f_{\mu_i}(x_k),\Delta_{i,k}\rangle -\alpha \norm{\nabla_i  f_{\mu_i}(x_k)}^2 \nonumber \\ & \quad +\frac{L_1\alpha^2}{2}\norm{G_{\mu_i}(x_k)}^2,\label{theorem_1_inequality}
    \end{align}
where $\Delta_{i,k}:=G_{\mu_i}(x_k)-\nabla_i f_{\mu_i}(x_k)$. The first equality follows by~(\ref{updating_rule}) and the fact that $\langle\nabla f_{\mu_i}(x_k),x_{k+1}-x_{k}\rangle=\langle\nabla_i f_{\mu_i}(x_k),x_{k+1}-x_{k}\rangle$, which holds since $x_{k+1}$ and $x_k$ only differ at block $i$. 
Taking expectation w.r.t. $u_{[k]}$ and $i_{[k-1]}$ on both sides of~\eqref{theorem_1_inequality} conditioned on the event $\{i_k=i\}$, we get that 
\begin{equation}\label{f_mu_i}
    \begin{split}
     & \mathbb{E}\left[ \norm{\nabla_i  f_{\mu_i}(x_k)}^2\right]\leq \frac{ \mathbb{E}\left[f_{\mu_i}(x_{k})\right] -\mathbb{E}\left[f_{\mu_i}(x_{k+1})\right]}{\alpha}\\
     &+\frac{L_1\alpha}{2}\mathbb{E}\left[\norm{G_{\mu_i}(x_k)}^2  \right],
    \end{split}
\end{equation}
where the inner-product term $\langle \nabla_i  f_{\mu_i}(x_k),\Delta_{i,k}\rangle$ disappears since $\mathbb{E}[\Delta_{i,k}]=0$ due to Lemma~\ref{unbiased_gradient}.
%
% Taking the expectation w.r.t. $u_{k}$ over the middle inner-product given $u_{[k-1]}$ and $i_{[k-1]}$, we have $\mathbb{E}_{u_{k}}\left[\langle \nabla_i  f_{\mu_i}(x_k),\Delta_{i,k}\rangle\right|u_{[k-1]},i_{[k-1]}]=0$ according to Lemma~\ref{unbiased_gradient}. 
According to Lemma~\ref{approximation_errors} and using the fact that $(a+b)^2\leq 2a^2+2b^2,$ we have
\begin{equation}\label{gradient_error_inequality}
    \norm{\nabla_i f(x)}^2 \leq 2\norm{\nabla_i f_{\mu_i}(x)}^2 + \mu_i^2{L_1}^2 (n_i+3)^{3}.
\end{equation}
Combining ({\ref{f_mu_i}}) and~(\ref{gradient_error_inequality}) and, we obtain that  
\begin{equation}\label{f_i_pre}
    \begin{split}
     & \frac{1}{2}\mathbb{E} \left[\norm{\nabla_i  f(x_k)}^2\right]\leq \frac{ \mathbb{E}\left[f_{\mu_i}(x_{k})\right] -\mathbb{E}\left[f_{\mu_i}(x_{k+1})\right]}{\alpha}\\
     &+\frac{L_1\alpha}{2}\mathbb{E}\left[\norm{G_{\mu_i}(x_k)}^2  \right]+\frac{1}{2}\mu^2{L_1}^2 (\Bar{n}+3)^{3},
    \end{split}
\end{equation}
where the last term follows by substituting the common smoothing parameter $\mu$ and agents' dimension $\Bar{n}$. Taking expectation on both sides of~\eqref{f_i_pre} w.r.t. $i_k$, we have that
\begin{equation}\label{f_i}
    \begin{split}
     & \frac{1}{2}\mathbb{E}_{i_k}\left[\mathbb{E}_{u_{[k]},i_{[k-1]}} \left[\norm{\nabla_i  f(x_k)}^2|i_k=i\right]\right]\\
     &\leq \frac{ \mathbb{E}\left[f_{\tilde{\mu}}(x_{k})\right] -\mathbb{E}\left[f_{\tilde{\mu}}(x_{k+1})\right]}{\alpha}\\
     &+\frac{L_1\alpha}{2}\mathbb{E}\left[\norm{G_{\tilde{\mu}}(x_k)}^2  \right]+\frac{1}{2}\mu^2{L_1}^2 (\Bar{n}+3)^{3},
    \end{split}
\end{equation}
where $\mathbb{E}[f_{\tilde{\mu}}(x_{k})]:=\mathbb{E}_{i_k}[\mathbb{E}_{u_{[k]},i_{[k-1]}}[f_{\mu_{i_k}}(x_k)|i_{k}=i]]$ and $\mathbb{E}[f_{\tilde{\mu}}(x_{k+1})]:=\mathbb{E}_{i_k}[\mathbb{E}_{u_{[k]},i_{[k-1]}}[f_{\mu_{i_k}}(x_{k+1})|i_{k}=i]].$ $\mathbb{E}[\norm{G_{\tilde{\mu}}(x_k)}^2]$ follows from the definition in Lemma~\ref{lemma_second_moments}. Next, we show that the left hand side of~\eqref{f_i} satisfies
\begin{align}\label{p_min_inequ}
  &\mathbb{E}_{i_k}[\mathbb{E}_{u_{[k]},i_{[k-1]}} [\norm{\nabla_{i_k}  f(x_k)}^2|i_k=i]] \nonumber\\
  &\geq p_{\min}\mathbb{E}_{u_{[k]},i_{[k]}}\norm{\nabla  f(x_k)}^2.  
\end{align}
To see this, by definitions of the projected gradient as in Section~\ref{formulation}, since $\nabla_i  f(x_k)$ is only nonzero at block $i$, we have $\norm{\nabla  f(x_k)}^2=\sum_{i=1}^N \norm{\nabla_i  f(x_k)}^2$. Therefore,
we can further get that
\begin{align}\label{sum_nabla}
  &\mathbb{E}_{u_{[k]},i_{[k]}}\norm{\nabla  f(x_k)}^2  =  \sum_{i=1}^N \mathbb{E}_{u_{[k]},i_{[k]}} \norm{\nabla_i  f(x_k)}^2\nonumber\\
  &=\sum_{i=1}^N\mathbb{E}_{u_{[k]},i_{[k-1]}} [\norm{\nabla_i  f(x_k)}^2|i_k=i],
\end{align}
where the second equality holds since $x_k$ is independent from $i_k$.
Therefore, according to (31), to show the inequality (30), it is sufficient to show $\mathbb{E}_{i_k}[\mathbb{E}_{u_{[k]},i_{[k-1]}} [\norm{\nabla_{i_k}  f(x_k)}^2|i_k=i]] \geq p_{\min}\sum_{i=1}^N\mathbb{E}_{u_{[k]},i_{[k-1]}} [\norm{\nabla_i  f(x_k)}^2|i_k=i]$. This is simple to prove because $\mathbb{E}_{i_k}[\mathbb{E}_{u_{[k]},i_{[k-1]}} [\norm{\nabla_{i_k}  f(x_k)}^2|i_k=i]] = \sum_{i} p_i \mathbb{E}_{u_{[k]},i_{[k-1]}} [\norm{\nabla_i  f(x_k)}^2|i_k=i].$ Therefore, inequality (30) is true.
Substituting~\eqref{p_min_inequ} into~\eqref{f_i} and then summing~\eqref{f_i} from $k=0$ to $T-1$, we have that
\begin{align}
     & \frac{p_{\min}}{2}\sum_{k=0}^{T-1}\mathbb{E} \left[\norm{\nabla  f(x_k)}^2\right] \leq  \frac{ \mathbb{E}\left[f_{\tilde{\mu}}(x_{0})\right]  -\mathbb{E}\left[f_{\tilde{\mu}}(x_{T})\right]}{\alpha} \nonumber\\
     &+\sum_{k=0}^{T-1}\frac{L_1\alpha}{2}\mathbb{E}\left[\norm{G_{\tilde{\mu}}(x_k)}^2  \right]+\frac{\mu^2}{2}{L_1}^2 (\Bar{n}+3)^{3} T\label{telescope_sum_2}.
\end{align}
% The inequality holds since for all $k\in\{1,\dots,T-1\},$ we have $\mathbb{E}_{u_{[k]},i_{[k]}}[f_{\tilde{\mu}}(x_k)] = \mathbb{E}_{u_{[k-1]},i_{[k-1]}}[f_{\tilde{\mu}}(x_k)]$ as $x_k$ is fixed given $u_{[k-1]},i_{[k-1]}.$ {\color{red} I do not understand this sentence, how is this relevant to get the previous inequality. Where do these two expectation appear in previous inequality?}
%
% Since for each $k$, we can bound the second moments of $G_\mu$ by past second moments as indicated in Lemma~\ref{lemma_second_moments}, we now need to show that the sum of second moments of gradient estimators from $0$ to $T-1$ can be bounded as well. Before showing that, we first prove the following lemma. 
%
%
% Applying the bound on $\sum_{k=0}^{T-1} \mathbb{E}\left[\norm{G_{\tilde{\mu}}(x_k)}^2  \right]$ in Lemma~{\ref{lemma_recursive_G}} to \eqref{telescope_sum_2}, we obtain that :
% %
% \begin{equation*}
%     \begin{split}
%     &\sum_{k=0}^{T-1}\mathbb{E}\left[\norm{G_{\tilde{\mu}}(x_k)}^2 \right]\leq  \frac{1-\beta}{1-(\gamma+\beta)}\mathbb{E}\left[\norm{G_{\tilde{\mu}}(x_0)}^2\right]\\
%     &+(T-1)\frac{1-\beta}{1-(\gamma+\beta)}M - \frac{\gamma}{\left(1-(\gamma+\beta)\right)^2}M,
%     \end{split}
% \end{equation*}
% where $\beta,\gamma,M$ are as defined in Lemma~\ref{lemma_recursive_G} {\color{red}Is above inequality different in anywhere from the one in Lemma 5? Why do we rewrite it here?}.
% where $\gamma =\frac{2\Bar{n}L_0^2\alpha^2(T-1)}{\mu^2} $, $\beta=1-p_{\min}$ and $M=4L_0^2\left((4+\Bar{n})^2+\Bar{n}^2\right)$ provided with $0<\gamma+\beta <1.$
%
Applying Lemma~\ref{lemma_recursive_G} to~(\ref{telescope_sum_2}), we get that 
\begin{align}
    & \frac{p_{\min}}{2}\sum_{k=0}^{T-1}\mathbb{E}\left[ \norm{\nabla  f(x_k)}^2 \right]\leq \nonumber\\
    &\frac{ \mathbb{E}\left[f_{\tilde{\mu}}(x_{0})\right]-\mathbb{E}\left[f_{\tilde{\mu}}(x_{T})\right]}{\alpha}+\frac{L_1\alpha}{2}(T-1)\frac{1-\beta}{1-(\gamma+\beta)}M \nonumber\\
    & +\frac{L_1\alpha}{2}\frac{1-\beta}{1-(\gamma+\beta)}\mathbb{E}\left[\norm{G_{\tilde{\mu}}(x_0)}^2\right]  +\frac{\mu^2}{2}{L_1}^2 (\Bar{n}+3)^{3} T\nonumber\\
    &  - \frac{L_1\alpha}{2}\frac{\gamma}{\left(1-(\gamma+\beta)\right)^2}M \label{final_inequa},
\end{align}
where $\gamma,\beta$ and $M$ are as defined in Lemma~\ref{lemma_recursive_G}.
Selecting $\mu=\frac{2L_0}{T^{\frac{1}{6}}}$ and $\alpha=\frac{\sqrt{p_{\min}}}{\sqrt{\Bar{n}}T^{\frac{2}{3}}}$, we have $\gamma\leq\frac{p_{\min}}{2}$ and $1-(\gamma+\beta)\geq \frac{p_{\min}}{2} $. Substituting these values into~\eqref{final_inequa} and omitting the negative term, we obtain that 
\begin{equation}\label{final_bound}
    \begin{split}
     &\frac{p_{\min}}{2}\sum_{k=0}^{T-1}\mathbb{E}\left[ \norm{\nabla  f(x_k)}^2 \right]\leq  \frac{ \mathbb{E}\left[f_{\tilde{\mu}}(x_{0})\right]-f_{\tilde{\mu}}^*}{\sqrt{p_{\min}}}\sqrt{\Bar{n}}T^{\frac{2}{3}}\\
     &+L_1\frac{\sqrt{p_{\min}}}{\sqrt{\Bar{n}}T^{\frac{2}{3}}}\mathbb{E}\left[\norm{G_{\tilde{\mu}}(x_0)}^2\right] +L_1\frac{\sqrt{p_{\min}}}{\sqrt{\Bar{n}}}T^{\frac{1}{3}}M\\
     &+2L_0^2\mu^2{L_1}^2 (\Bar{n}+3)^{3} T^{\frac{2}{3}},
    \end{split}
\end{equation}
where $f_{\tilde{\mu}}^*$ is a lower bound on $\mathbb{E}[f_{\tilde{\mu}}(x)]$. The existance of such lower bound is due to~(\ref{function_error}), the definition of $\mathbb{E}[f_{\tilde{\mu}}(x)]$ and Assumption~\ref{lipschitz}.
%{\color{red}we need to assume that the original objective function is lower bounded first. I did not check the selection of the parameters and the arithmetics here. Please double check.}. 
%
The result in Theorem~\ref{theorem_1} follows by dividing both sides of the above inequality by $T$. 
\end{proof}
The convergence rate shown above has the same order as that of applying the residual-feedback gradient estimator~\eqref{residual-feedback} to optimize a stochastic objective function as shown in~\cite{zhang2020improving}. This is because of this asynchronous scenario, the updates conducted by the other agents between two queries of a given agent introduce noise in the function evaluations from the perspective of this given agent. Furthermore, the bound on the non-stationarity of the solution in~\eqref{final_bound} increases as $p_{\min}$ becomes smaller. In practice, it means that the convergence of Algorithm~\ref{alg:B} slows down if one of the agents is activated less frequently than others.  
%%%%%%%%%%%%%%%%%%%%%%%%%%%%%%%%%
\section{SIMULATIONS}\label{simulation}
In this section, we demonstrate the effectiveness of the proposed asynchronous distributed zeroth-order optimization algorithm on a distributed feature learning example common in Internet of Things (IoT) applications. All the experiments are conducted using Python 3.8.5 on a 2017 iMac with 4.2GHz Quad-Core Intel Core i7 and 32GB 2400MHz DDR4.

Specifically, we consider the biomarker learning example described in \cite{bent2020digital}, where a network of health monitoring edge devices collect heterogeneous raw input data $\{D_{i,j}\}_{i=1:N}$, e.g., different types of biosignals. Then, each device encodes its local raw data $D_{i,j}$ into a biomarker $d_{i,j}$ via a feature extraction function $\phi(D_{i,j}; x_i)$, e.g., a neural network with weights $x_i$, and sends it to a third-party entity that uses the collected biomarkers as predictors to learn a disease diagnosis for user $j$. The goal of the edge device $i$ is to learn a better feature extraction function $\phi(\cdot ; x_i)$ to help the third-party entity to make better predictions.
In practice, the prediction process at the third-party entity can be complicated and hard to model using an explicit function. Moreover, it may need to remain confidential. As a result, the edge device cannot obtain gradient information from this third-party entity. In the meantime, it is unreasonable to expect that the edge devices can update their feature extraction models synchronously or that they know the other edge devices' feature extraction models and parameters. Therefore, this problem presents an ideal case for the asynchronous distributed zeroth-order method proposed in this paper.

For simulation proposes, in this section, we assume that the third-party entity uses the logistic regression model
\begin{align} \label{eqn:log_model}
    P(y_j ; d_{j}) = 1/\big(1 + \exp{(-y_i W^T d_j)}\big),
\end{align}
where $j$ represents the data point index, $y_j = \{1, -1\}$ and $d_j$ denote the label and predictors for data point $j$, and $W_T$ is a fixed classifier parameter. Specifically, let $d_j = [d_{1,j}, \dots, d_{N,j}]^T$ represent the concatenated biomarker vector. The agents aim to collaboratively minimize the following loss function
\begin{align} \label{eqn:log_loss}
    f(\{x_i\}_{i=1:N}) = -\frac{1}{J}\sum_{j = 1}^{J}\log\big(P(y_j; d_j(\cdot ; \{x_i\}_{i=1:N})\big),
\end{align}
where $J$ is the total number of data points.

Next, we apply the proposed Algorithm~\ref{alg:B} to this distributed feature learning problem and compare its performance to an asynchronous extension of the centralized two-point gradient estimate~\eqref{two_point_unbiased} defined by
\begin{equation}\label{two_point_as}
     G_{\mu_i}(x_k) =  \frac{f(x_k+\mu_i u_k)-f(x_{k-M})}{\mu_i}u_k,
\end{equation}
where $x_{k-M}$ is the most recent decision point agent $i$ queries at iteration $k-M$. 

Note that the convergence of the stochastic gradient descent update \eqref{updating_rule} using the asynchronous two-point estimator~\eqref{two_point_as} has not been studied yet. We compare our proposed algorithm to the one with \eqref{two_point_as} simply to demonstrate the efficacy of our proposed approach.
%%%%%%%%%%%%
Specifically, we consider a network of $5$ agents who collaboratively deal with $J=20$ data samples. The feature extraction model $\phi(\cdot; x_i)$ at agent $i$ is a single layer neural network with the input $D_{i,j}\in\mathbb{R}^{10}$ and a single output $d_{i,j}\in\mathbb{R}$. The activation function of the neural network is the sigmoid function. The weight of the neural network at agent $i$ is denoted as $x_i$, which is initialized by sampling from a standard Gaussian distribution.
% $\phi(D_{i,j}, x_i):= x_i^T D_{i,j}$, where $x_i \in \mathbb{R}^{10}$. The raw data $\{D_{i,j}\}$, labels $\{y_j\}$ and prediction model parameter $W$ are randomly generated. 
We apply both the asynchronous residual-feedback gradient estimator~\eqref{residual} and the asynchronous two-point gradient estimator~({\ref{two_point_as}}) to solve this problem. 

Specifically, for both gradient estimators, we run $10$ trials. In addition, the smoothing parameter is $\mu = 0.1$, and the stepsizes $\alpha$ for gradient estimators \eqref{residual} and ~({\ref{two_point_as}}) are selected as $0.5$ and $0.5$, respectively, so that they both achieve their fastest convergence speed during $10$ trials of experiments. At each iteration, each agent has equal probability to be activated.
\begin{figure}[t]
    \centering
    \includegraphics[width = .95\columnwidth]{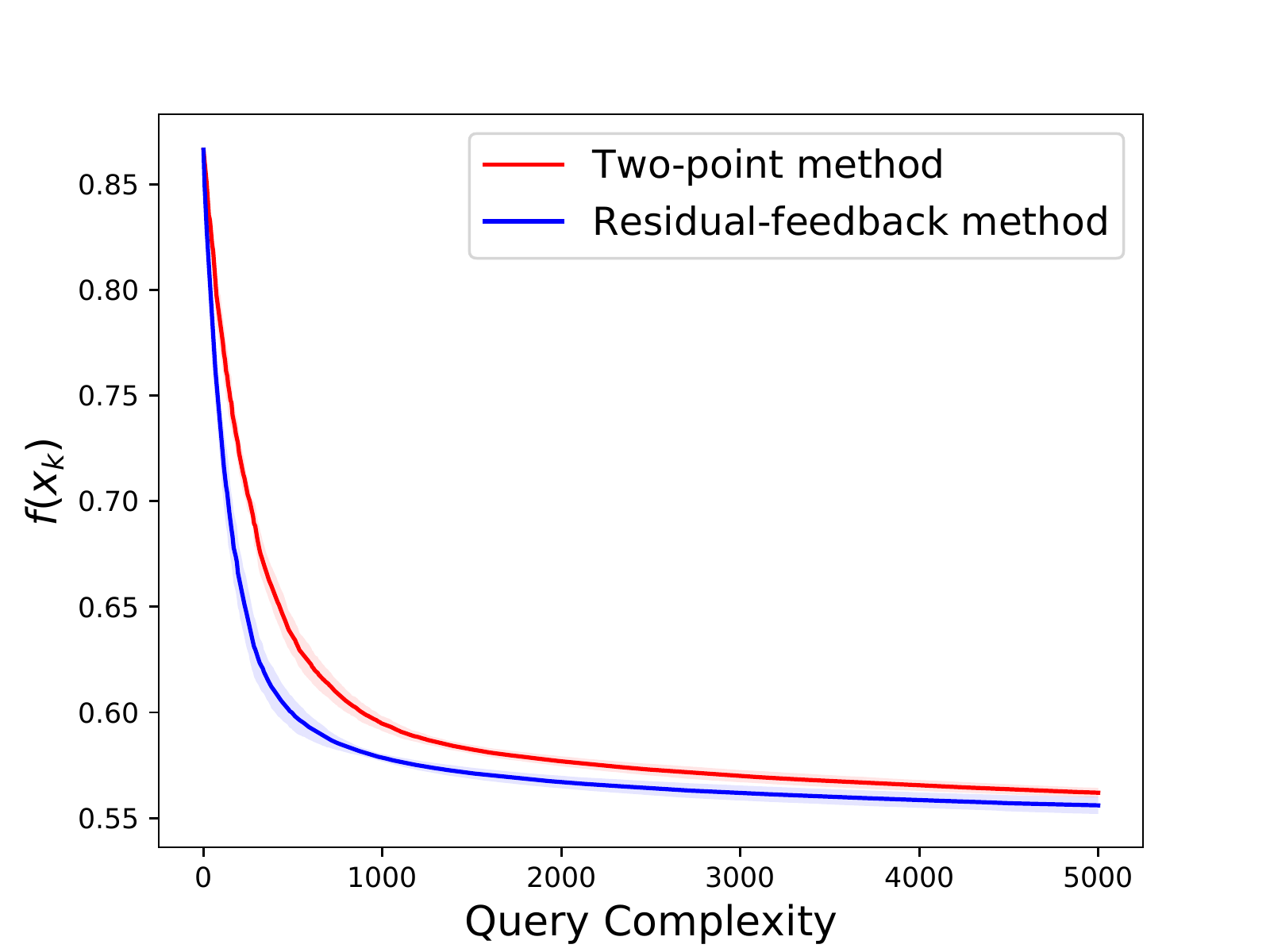}
    \caption{Convergence results of the distributed feature learning problem. The red curve is obtained by applying the asynchronous two-point gradient estimator~({\ref{two_point_as}}) and the blue curve is by the asynchronous residual-feedback estimator~\eqref{residual}. The y axis denotes the value of the loss function~\eqref{eqn:log_loss} and the x axis represents the number of queries made in total by the team of agents. The shaded area around each curve represents the standard deviation of the function values over $10$ trials.}
    \label{fig:log_regression}
\end{figure}

The comparative performance results of using the two zeroth-order gradient estimators~\eqref{residual} and~({\ref{two_point_as}}) are presented in Figure~\ref{fig:log_regression}. We observe that during $10$ trials, asynchronous learning with the residual-feedback gradient estimator~\eqref{residual} converges faster than asynchronous learning with the two-point gradient estimator~({\ref{two_point_as}}). This is because the asynchronous residual-feedback gradient estimator~\eqref{residual} is subject to almost the same level of variance as the two-point gradient estimator~({\ref{two_point_as}}), but can make twice the number of updates compared to the two-point gradient estimator~({\ref{two_point_as}}) for the same number of queries. Note that we compare the two algorithms in terms of the number of queries rather than the number of updates, because the number of queries corresponds to the length of the global wall time required to run the algorithm.

\section{CONCLUSIONS}\label{conclusion}
In this paper, we proposed an asynchronous residual-feedback gradient estimator for distributed zeroth-order optimization, which estimates the gradient of the global cost function by querying the value of the function once at each time step. More importantly, only the local decision vector is needed for estimating the gradient and no communication among agents is required. We showed that the convergence rate of the proposed method matches the results for centralized residual-feedback methods when the function evaluation has noise. Numerical experiments on a distributed logistic regression problem are presented to show the effectiveness of the proposed method.   

\addtolength{\textheight}{-3cm}   % This command serves to balance the column lengths

%%%%%%%%%%%%%%%%%%%%%%%%%%%%%%%%%%%%%%%%%%%%%%%%%%%%%%%%%%%%%%%%%%%%%%%%%%%%%%%%
%%%%%%%%%%%%%%%%%%%%%%%%%%%%%%%%%%%%%%%%%%%%%%%%%%%%%%%%%%%%%%%%%%%%%%%%%%%%%%%%
%%%%%%%%%%%%%%%%%%%%%%%%%%%%%%%%%%%%%%%%%%%%%%%%%%%%%%%%%%%%%%%%%%%%%%%%%%%%%%%%
\section*{APPENDIX}

\begin{lm}\label{lemma_recursive}
Consider a sequence of non-negative real numbers $\{V_k\}$ with the following relations for all $1\leq k \leq T-1,$ provided with $0<\gamma+\beta <1$,
$$V_k \leq \gamma \left(V_{k-1} + \beta V_{k-2}+ \dots + \beta^{k-1} V_0\right)+M,$$ where $M$ is a contant, then we have 
$$V_k \leq \gamma(\gamma +\beta)^{k-1}V_0+\frac{1-\beta -\gamma(\gamma+\beta)^{k-1}}{1-(\gamma+\beta)}M.$$ In addition,
\begin{equation*}
    \begin{split}
    \sum_{k=0}^{T-1} V_k \leq & \frac{1-\beta}{1-(\gamma+\beta)}V_0+(T-1)\frac{1-\beta}{1-(\gamma+\beta)}M \\
      &- \frac{\gamma}{\left(1-(\gamma+\beta)\right)^2}M .
    \end{split}
\end{equation*}
\end{lm}
% \begin{proof}
% \end{proof}
\begin{proof}
Fix some $k=K$. We have
\begin{equation*}
    \begin{split}
      V_{K} \leq &\gamma \left(V_{K-1} + \beta V_{K-2}+ \dots + \beta^{K-1} V_0\right)+M \\
    %   \leq & \gamma V_{K-1} + \gamma\beta\left(V_{K-2}+\dots+\beta^{K-2} V_0 \right) + M \\
    %   \leq & \gamma^2 \left(V_{K-2}+\dots+\beta^{K-2} V_0 \right) \\
    %   &+ \gamma\beta\left(V_{K-2}+\dots+\beta^{K-2} V_0 \right)\\
    %   &+ \gamma M + M \\
      %
      \leq & \gamma(\gamma+\beta) \left(V_{K-2}+\dots+\beta^{K-2} V_0 \right) +\gamma M + M \\
    \end{split}
\end{equation*}
Repeat the above process, we obtain that 
\begin{equation*}
    \begin{split}
    %   V_{K} \leq & \gamma(\gamma+\beta)^{K-2} \left(V_{1}+\beta V_0 \right) +\gamma \sum_{k=0}^{K-3}(\gamma+\beta)^k M  +M\\
    %   %
      V_{K} \leq & \gamma(\gamma+\beta)^{K-2} \left(\gamma V_{0}+M+\beta V_0 \right) \\
      &+\gamma \sum_{k=0}^{K-2}(\gamma+\beta)^k M  +M \\
      %
    %   =& \gamma(\gamma+\beta)^{K-1}V_0+ \frac{\gamma\left(1-(\gamma+\beta)^{K-1}\right)}{1-(\gamma+\beta)}M+M\\
      =& \gamma(\gamma+\beta)^{K-1}V_0+ \frac{1-\beta -\gamma(\gamma+\beta)^{K-1}}{1-(\gamma+\beta)}M,
    \end{split}
\end{equation*}
which completes the first part of the proof. Summing $V_k$ from $0$ to $T-1$, we obtain that  
\begin{equation*}
    \begin{split}
      \sum_{k=0}^{T-1} V_k =& \sum_{k=1}^{T-1}\left( \gamma(\gamma+\beta)^{k-1}V_0+ \frac{1-\beta -\gamma(\gamma+\beta)^{k-1}}{1-(\gamma+\beta)}M\right) \\
      & + V_0\\
    %   \leq & \frac{\gamma}{1-(\gamma+\beta)}V_0+(T-1)\frac{1-\beta}{1-(\gamma+\beta)}M \\
    %   &- \frac{\gamma}{\left(1-(\gamma+\beta)\right)^2}M + V_0\\
      = & \frac{1-\beta}{1-(\gamma+\beta)}V_0+(T-1)\frac{1-\beta}{1-(\gamma+\beta)}M \\
      &- \frac{\gamma}{\left(1-(\gamma+\beta)\right)^2}M,
    \end{split}
\end{equation*}
which complete the proof of the lemma. 
\end{proof}
% \section*{ACKNOWLEDGMENT}

% The preferred spelling of the word acknowledgment in America is without an e after the g. Avoid the stilted expression, One of us (R. B. G.) thanks . . .  Instead, try R. B. G. thanks. Put sponsor acknowledgments in the unnumbered footnote on the first page.

%%%%%%%%%%%%%%%%%%%%%%%%%%%%%%%%%%%%%%%%%%%%%%%%%%%%%%%%%%%%%%%%%%%%%%%%%%%%%%%%

\bibliographystyle{IEEEtran}
\bibliography{root}

%%%%%%%%%%%%%%%%%%%%%%%%%%%%%%%%%%%%%%%%%%%%%%%%%%%%%%%%%%%%%%%%%%

\end{document}